\documentclass[reqno]{amsart}

\usepackage{amsmath,amsthm,amssymb}

\usepackage{graphicx}
\usepackage{microtype}

\newcommand{\Prob}{\mathbb{P}}
\newcommand{\BBE}{\mathbb{E}}
\newcommand{\BBC}{\mathbb{C}}
\newcommand{\BBR}{\mathbb{R}}

\newcommand{\cA}{\mathcal{A}}

\newcommand{\cH}{\mathcal{H}}

\newcommand{\cF}{\mathcal{F}}

\newcommand{\BA}{\mathbf{A}}

\newcommand{\BM}{\mathbf{M}}
\newcommand{\Id}{\mathbf{I}}
\newcommand{\BX}{\mathbf{X}}
\newcommand{\BY}{\mathbf{Y}}
\newcommand{\BH}{\mathbf{H}}
\newcommand{\BW}{\mathbf{W}}
\newcommand{\BSigma}{\mathbf{\Sigma}}

\newcommand{\Ave}{\mathbf{A}}
\newcommand{\Har}{\mathbf{H}}

\newcommand{\fHar}{\mathfrak{h}}
\newcommand{\fPois}{\mathfrak{p}}
\newcommand{\fCovar}{\mathfrak{s}}
\newcommand{\fError}{\mathfrak{e}}

\DeclareMathOperator{\Tr}{Tr}

\DeclareMathOperator{\diff}{d\!}

\newtheorem{theorem}{Theorem}[section]
\newtheorem{corollary}{Corollary}[theorem]

\newtheorem{lemma}[theorem]{Lemma}

\theoremstyle{definition}
\newtheorem{definition}{Definition}

\newtheorem{remark}{Remark}

\title{Harmonic Means of Wishart Random Matrices}
\author{Asad Lodhia}
\address{256 West Hall, 1085 South University Avenue, Ann Arbor MI, 
48109-1107}
\email{alodhia@umich.edu}
\begin{document} 

\begin{abstract}
We use free probability to compute the limiting spectral properties of
the harmonic mean of $n$ i.i.d.\ Wishart random matrices $\BW_i$ whose
limiting aspect ratio is $\gamma \in (0,1)$ when $\BBE[\BW_i] =
\Id$. We demonstrate an interesting phenomenon where the harmonic mean
$\Har$ of the $n$ Wishart matrices is closer in operator norm to
$\BBE[\BW_i]$ than the arithmetic mean $\Ave$ for small $n$, after
which the arithmetic mean is closer. We also prove some results for
the general case where the expectation of the Wishart matrices are not
the identity matrix.
\end{abstract}

\maketitle

\section{Introduction}

Positive definite random matrices are often studied in probability 
theory and statistics. The most famous (and arguably most widely used) 
matrix model supported on the set of positive semidefinite matrices is 
the Wishart ensemble.  Let $\{\BX_i\}_{i=1}^n$ be a sequence of 
centered independent identically distributed matrices that have 
dimension $P\times N$ whose entries have at least two finite moments. 
Suppose each column of $\BX_i$ is an independent $P$-dimensional random 
vector. The matrices
\[
\BW_i:=\frac{\BX_i \BX_i^*}{N}
\]
are called Wishart matrices. If the columns of each $\BX_i$ are i.i.d.\ 
observations from a Gaussian distribution it suffices to specify their 
covariance matrix $\BSigma = \BBE[\BW_i]$ to obtain their distribution. 
In statistics the estimation of such a covariance matrix is a 
fundamental task. Our interest in this paper will be the mathematical 
study of estimates in operator norm of the covariance in the 
high-dimensional regime $\frac{P}{N} \to \gamma \in (0,1)$.

The notational choice in the previous paragraph may seem odd to the 
reader. If the columns of $\BX_i$ are drawn i.i.d., one may combine 
them, say by computing the arithmetic mean
\begin{equation}
\Ave := \frac{\sum_{i=1}^n \BW_i}{n}.
\end{equation}
This reduces the variance by a factor of $n^{-1}$ and is 
equivalent to adjoining the columns of the $\BX_i$ into a single 
$P$-by-$Nn$ matrix, since
\[
\Ave = \frac{\big[ \BX_1, \cdots, \BX_n\big] \big[\BX_1,\cdots, 
\BX_n\big]^*}{Nn}.
\]

In the regime where $P/N\to \gamma\in(0,1)$ the sample 
covariance matrix $\BW_i$ does not converge to its expected value 
$\BSigma$. Instead, when $\BBE \BW_i = \Id$, the spectral measure of 
each $\BW_i$ satisfies the Mar\v{c}enko-Pastur Law with parameter 
$\gamma$:
\begin{equation}
\label{eqn:mpdef}
\rho_{\mathrm{MP},\gamma}(x) := \frac{\sqrt{\big((1 + \sqrt{\gamma})^2
    - x\big)\big(x - (1 - \sqrt{\gamma})^2\big)}}{2\pi\gamma x}
\mathbf{1}_{\big[(1-\sqrt{\gamma})^2,(1+\sqrt{\gamma})^2\big]}(x).
\end{equation}
In fact, under sufficient moment conditions \cite{YBK88}, we have the
stronger result that
\[
\|\BW_i - \Id \| \to \gamma + 2\sqrt{\gamma} \quad \mathrm{a.s.},
\]
where $\|\BM\|$ represents the operator norm of the matrix $\BM$. It
is important to note here that the value of the operator norm in this
particular case is due to the right-edge of the spectrum of the
Mar\v{c}enko-Pastur Law. Subtraction of the matrix $\Id$ shifts all of
the eigenvalues of $\BW_i$ exactly by one and the eigenvalue with
largest absolute value is at the right edge of the spectrum.
Heuristically, our error is due to overestimating the largest
eigenvalue. When we average the $\BW_i$ resulting operator norm bound
becomes
\[
\|\Ave - \Id\| \to
\frac{\gamma}{n} + 2\sqrt{\frac{\gamma}{n}}\qquad \hbox{a.s.}
\]
The above limit follows from our interpretation of the arithmetic mean 
as a sample covariance matrix with aspect ratio $P/Nn \to \gamma/n$.  
Notice that the change in the operator norm error is not simply a 
rescaling by $n^{-1/2}$, even though the entrywise variance has changed 
by $n^{-1}$. The purpose of this paper is to explore an alternative to 
the arithmetic mean that takes into account the positive definiteness 
of $\BW_i$ when $P<N$.

The space of positive definite matrices is a cone and has a natural 
partial ordering. When $\BM_1$ and $\BM_2$ are $P\times P$ positive 
semidefinite matrices, we say
\[
\BM_1 \preceq \BM_2
\]
if and only if $\BM_2 - \BM_1$ is positive semidefinite. Under this
ordering one can show various generalizations of classical
inequalities.  Of particular interest in this paper, if $\BM_1$,
$\ldots$, $\BM_k$ are positive definite (and therefore invertible),
the classic arithmetic mean harmonic mean (AMHM) generalizes as
\cite[Theorem 1]{ST94}
\begin{equation}
\label{eqn:matamhmineq}
k\Big(\BM_1^{-1} + \cdots +  \BM_k^{-1}\Big)^{-1} \preceq \frac{\BM_1 + 
\cdots +\BM_k}{k}.
\end{equation}
The matrix on the left,
\[
k\Big(\BM_1^{-1} + \cdots + \BM_k^{-1}\Big)^{-1},
\]
is the harmonic mean of $\BM_1$, $\ldots$, $\BM_k$. This paper shows 
that $\Ave$ can give \emph{worse} estimates in operator norm than the 
matrix harmonic mean
\begin{equation}
\label{eqn:harmdef}
\Har := n \bigg(\sum_{i=1}^n \BW_i^{-1}\bigg)^{-1}.
\end{equation}
When $\BBE[\BW_i] = \Id$, we show for any $\gamma \in (0,1)$ and the 
operator norm of $\Har - \Id$ is always smaller than $\Ave - \Id$ when 
$n=2$. For general $n\geq 2$, this advantage disappears when $n$ 
exceeds a critical value $n^*(\gamma)$ that is a function only of 
$\gamma$.

A heuristic explanation of this result is the AMHM inequality $\Har 
\preceq \Ave$. We know from our discussion above that $\Ave$ is, in 
some sense, an overestimate of its expectation $\Id$. By taking a 
matrix smaller in the positive definite cone, we are compensating for 
this overestimation. As will be shown below, $\|\Har-\Id\|$ will be the 
absolute value of the \emph{smallest} eigenvalue of $\Har-\Id$, so 
$\Har$ \emph{underestimates} $\Id$. When $n$ is large, the spectral 
measure of $\Har$ approaches a point mass at $(1-\gamma)$ 
whereas the spectral measure of $\Ave$ approaches a point mass 
at $1$ (the spectral measure of $\Id$). This explains why 
eventually, for $n$ large enough, $\Ave$ is a better estimate.

The analysis presented in this paper
is complete for the case where $\BBE[\BW_i]= \Id$ but 
we will be able to comment on Wishart matrices with general non-singular 
covariance, by the observation that if $\BBE[\BW_i] = \Id$, then
\[
\BBE\big[\BSigma^\frac{1}{2}\BW_i \BSigma^{\frac{1}{2}}\big] = 
\BSigma.
\]
This fact implies that
for both the arithmetic and harmonic mean we simply need to multiply 
on both sides by $\BSigma^{1/2}$ to get the arithmetic and 
harmonic mean of a Wishart matrix with a general covariance $\BSigma$.
With some conditions on $\BSigma$, we can ensure that
the result sketched above still holds in this more general case.

\iffalse
We believe these results warrant further investigation in a more 
general setting. For example, suppose we have a sequence of matrices 
$\{\BSigma^{(i)}\}_{i=1}^n$ be a sequence of independent identically 
distributed self-adjoint positive definite matrices with dimension 
$P\times P$ and we wish to  estimate the mean of these matrices 
$\BSigma = \BBE\BSigma^{(i)}$. The standard way of doing this is to 
compute the sample mean
\[
\frac{1}{n}\sum_{i=1}^n \BSigma^{(i)},
\]
When $P$ is fixed and the entries of $\BSigma^{(i)}$ have finite
moments, as $n\to\infty$ the multivariate law of large numbers
guarantees the above average converges to the true expectation
$\BSigma$ in any matrix norm.  But what if $P$ is going to infinity
and the above sample mean is no longer convergent. Is there anything
better we can do?  \fi

\subsection*{Notation} In this paper $\Id$ will be the identity matrix, 
its dimension will be clear from the context. For a matrix $\BM$,
$\|\BM\|$ will always denote its operator norm and $\BM^*$ its
conjugate transpose. Given a set $A$, the function $\mathbf{1}_A(x)$
is the indicator function associated to that set. For a unital
$C^*$-algebra $\cA$, the norm will be denoted $\|\cdot\|_\cA$, the
unit element will be denoted $1_\cA$ and $*$ will denote the
involution. 

\subsection*{Acknowledgements} We are grateful to Alice Guionnet, 
Elizaveta Levina and Jinho Baik for their helpful comments and
suggestions.  We are also extremely grateful to Keith Levin for
reading earlier drafts of the paper and providing helpful comments.
This research was supported through NSF Grant DMS-1646108.
\section{Results and Outline}

For what follows, we will make the following assumption on the
matrices $\BX_i$ that generate $\BW_i$. We need these assumptions
primarily due to our application of Theorem~\ref{thm:strongconv}.
\begin{definition}[Matrix Model]
\label{def:matmodel}
The matrices $\{\BX_i\}_{i=1}^n$ are $P$ by $N$ and their entries are 
i.i.d.\ standard complex Gaussians\footnote{A standard complex Gaussian 
is of the form $\frac{Z_1 + \sqrt{-1} Z_2}{\sqrt{2}}$ where $Z_1$ and 
$Z_2$ are independent standard real Gaussian random variables.} and
\[
\Bigg| \frac{P}{N} - \gamma \Bigg| \leq \frac{K}{P^2},
\]
where $K > 0$ and $\gamma \in (0,1)$ are constants that do not depend
on $P$, $N$ or $n$. For each $i=1,2,\dots,n$,
define $\BW_i = N^{-1}\BX_i \BX_i^*$.
\end{definition}
\begin{figure}
\centering
\includegraphics[height=3in,width=4in]{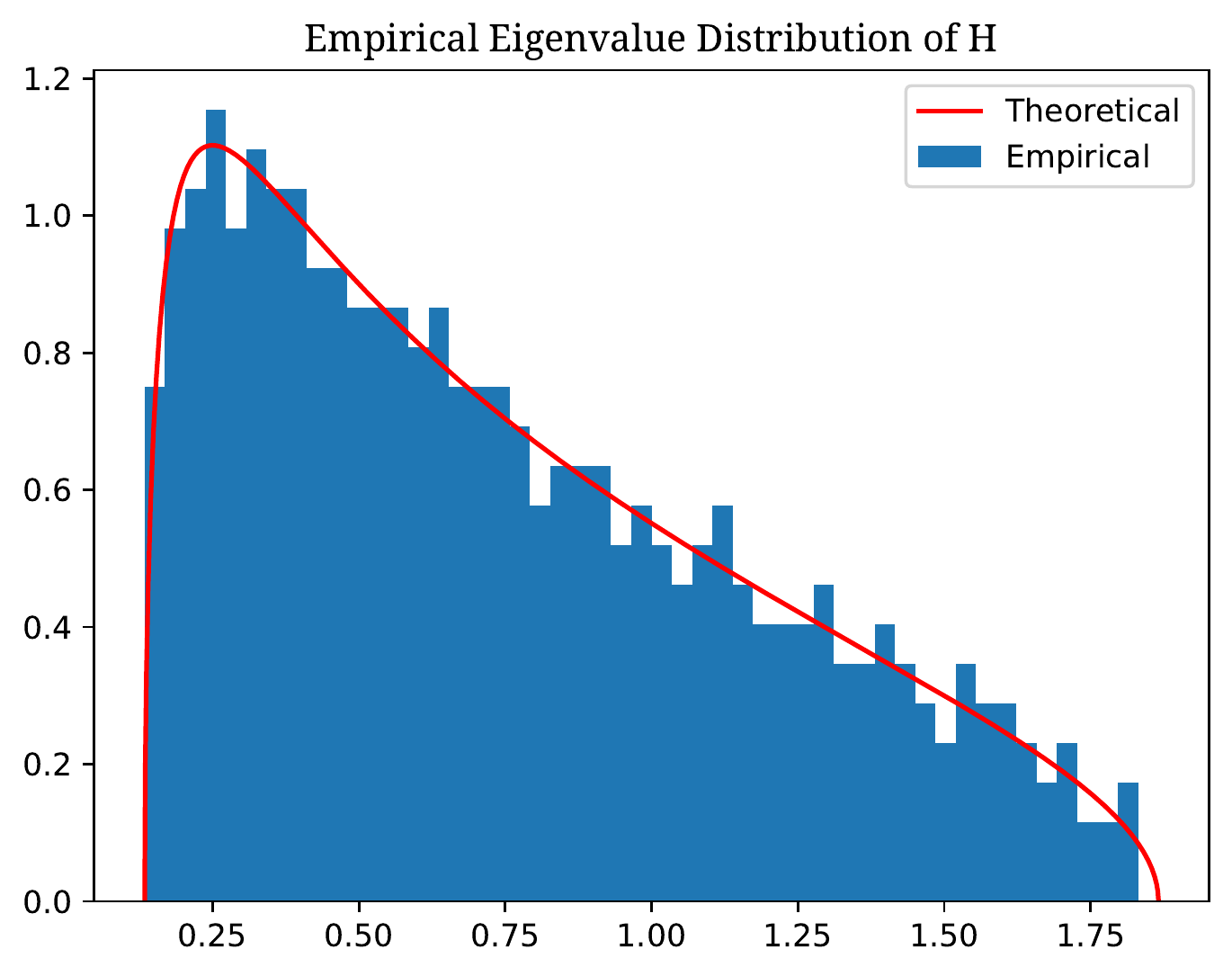}
\caption{\label{fig:harmesd0.5} A normalized histogram of a single 
sample of the Harmonic mean's empirical eigenvalue distribution with 
$P=500$, $N=1000$ and $n=2$ compared to the corresponding limiting 
distribution for $\gamma = 0.5$ and $n=2$.}
\end{figure}
We will prove the following result, which shows the harmonic mean of
Wishart random matrices can be closer in operator norm to the true
covariance than is the operator norm of the arithmetic mean. See
Figure~\ref{fig:harmesd0.5} for a simulation.
\begin{theorem}
\label{theorem:main}
Let $\BW_1$, $\ldots$, $\BW_n$ satisfy Definition~\ref{def:matmodel}.
Then for each fixed $n \geq 2$, the spectral measure of $\Har$ converges
weakly almost surely to the measure with density, i.e.,
\[
\frac{n}{2\pi \gamma x}\sqrt{(e_+ - x)(x-e_-)}\mathbf{1}_{[e_-,e_+]}(x),
\]
where
\[
e_\pm := 1 - \gamma + \frac{2\gamma}{n} \pm 2\sqrt{\frac{\gamma}{n}}
\sqrt{1 - \gamma + \frac{\gamma}{n}}.
\]
Further, we have the convergence:
\[
\lim_{P,N\to\infty}\|\Har - \Id\| \to \gamma - \frac{2\gamma}{n} +
2\sqrt{\frac{\gamma}{n}}\sqrt{1 - \gamma + \frac{\gamma}{n}} \quad
\mathrm{a.s.}
\]
\end{theorem}
\begin{remark}
\label{rem:cutoff}
Note that for small $n$
\[
\lim_{P,N \to \infty}\|\Har - \Id\| = 
\gamma-\frac{2\gamma}{n}+2\sqrt{\frac{\gamma}{n}}\sqrt{1 - \gamma + 
\frac{\gamma}{n}} < \frac{\gamma}{n} + 2\sqrt{\frac{\gamma}{n}} = 
\lim_{P,N\to \infty}\|\Ave - \Id\|,
\]
which is lost after $n$ exceeds a threshold $n^*(\gamma)$. Indeed,
the inequality is always true for $n=2$, where it reads:
\[
\sqrt{2\gamma}\sqrt{1-\frac{\gamma}{2}} < \frac{\gamma}{2} +
\sqrt{2\gamma},
\]
see Figure~\ref{fig:errorcompare} for a comparison of these functions, 
and notice the improvement of $\Har$ is larger as $\gamma$ gets closer 
to 1. Observe that as $n\to \infty$, $e_\pm$ converge to $1- \gamma$ 
which suggests that $\Har$ is somehow ``shrunken'' compared to $\Ave$. 
As mentioned in the Introduction,
\[
\lim_{P,N\to \infty}\|\Har - \Id\|  = 1- e_-
\]
so $\Har$ is off by the identity due to an ``underestimate'' of the
operator norm. 
\end{remark}

\begin{figure}
\centering
\includegraphics[height=3in,width=4in]{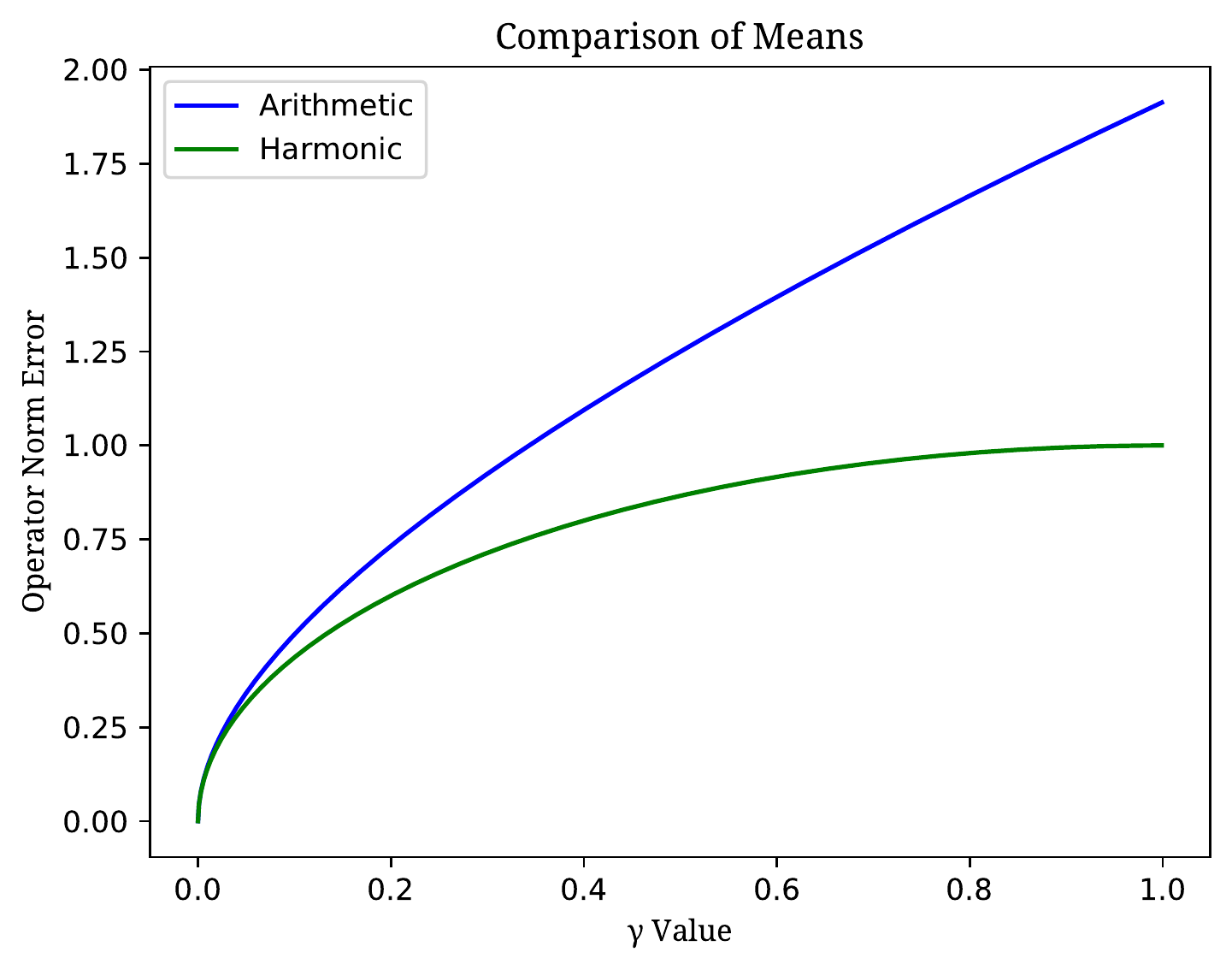}
\caption{\label{fig:errorcompare} A plot comparing the limiting
operator norm error of $\|\Har - \Id\|$ with $\|\Ave -\Id\|$ as 
a function of $\gamma \in (0,1)$ in the case that $n=2$.
}
\end{figure}

The Theorem~\ref{theorem:main} applies to matrices from 
Definition~\ref{def:matmodel}. For applications to statistics and other 
fields, it may be more desirable to have a model for general 
subgaussian real random matrices.

\begin{definition}[Alternative Matrix Model]
\label{def:altmatmodel}
The matrices $\{\BX_i\}_{i=1}^n$ are $P$ by $N$ and their entries are 
i.i.d.\ real subgaussian \footnote{A centered real subgaussian random 
variable $X$ is a random variable such that there exists a $\sigma > 0$ 
such that \[ \BBE[\exp(tX)] \leq \exp\bigg(\frac{\sigma^2 t^2}{2}\bigg) 
\] for  all $t \in \BBR$. The number $\sigma$ is often called the 
subgaussian parameter of $X$.} and
\[
\Bigg| \frac{P}{N} - \gamma \Bigg| \leq \frac{K}{P^2},
\]
where $K > 0$ and $\gamma \in (0,1)$ are constants that do not depend 
on $P$, $N$ or $n$. For each $i=1,2,\dots,n$, define $\BW_i = 
N^{-1}\BX_i \BX_i^*$. 
\end{definition}

A few of the Lemma used to prove Theorem~\ref{theorem:main} carry 
through to the matrices in the Definition~\ref{def:altmatmodel}. This 
strongly suggests Theorem~\ref{theorem:main} should hold for more 
general assumptions on the matrix entries. See 
Remark~\ref{rem:generalization} for a technical discussion that 
clarifies this possible extension.

Another natural question is whether the results above carry over to the 
case where $\BBE[\BW_i]\neq \Id$. A simple submultiplicativity argument 
combined with the above Theorem gives the following result:
\begin{corollary}
\label{cor:submult}
Assume $n\leq n^*(\gamma)$. Let $\BSigma$ be a sequence of deterministic
$P\times P$ positive definite covariance matrices (with $P$-dependence
suppressed) such that 
\[
\limsup_{P,N\to\infty}\frac{ \|\BSigma\|\|\BSigma^{-1}\|\|\Har - \Id\|}
       {\|\Ave- \Id\|}<1,
\]
then
\[
\limsup_{P,N\to\infty}
\frac{\|\BSigma^{\frac{1}{2}}\Har\BSigma^{\frac{1}{2}} -
  \BSigma\|}{\|\BSigma^{\frac{1}{2}}\Ave\BSigma^{\frac{1}{2}} -
  \BSigma\|} < 1 \qquad \mathrm{a.s.}
\]
\end{corollary}
\begin{proof}
Since $\BSigma$ are positive definite, we have 
\[
\big\|\BSigma^\frac{1}{2}\big\|^2 = \|\BSigma\| \quad
\hbox{and}\quad \big\|\BSigma^{-\frac{1}{2}}\big\|^2 =
\big\|\BSigma^{-1}\big\|.
\]
Now, by submultiplicativity of the operator norm, it follows that
\[
\big\|\BSigma^{\frac{1}{2}}\Har\BSigma^{\frac{1}{2}} -
\BSigma\big\| \leq \big\|\BSigma\big\| \|\Har - \Id\| \leq
\big\|\BSigma\big\|\big\|\BSigma^{-1}\big\| \frac{\|\Har -
  \Id\|}{\|\Ave - \Id\|}
\|\BSigma^{\frac{1}{2}}\Ave\BSigma^{\frac{1}{2}} - \BSigma\|,
\]
since with probability one $\BA \neq \Id$ we know the quantity on the
right is non-zero. Hence we can rearrange to obtain the inequality
\[
\frac{\|\BSigma^{\frac{1}{2}}\Har \BSigma^{\frac{1}{2}} -\BSigma\|}
     {\|\BSigma^{\frac{1}{2}}\Ave\BSigma^{\frac{1}{2}} - \BSigma\|}
\leq \frac{\|\BSigma\|\|\BSigma^{-1}\| \|\Har - \Id\|}{\|\Ave-\Id\|},
\]
now taking the $\limsup$ of both sides yields the required result.
\end{proof}
\begin{remark}
The quantity $\|\BSigma\|\|\BSigma^{-1}\|$ is the largest eigenvalue of 
$\BSigma$ divided by the smallest eigenvalue of $\BSigma$. In 
applications, this is often called the \emph{condition number} of 
$\BSigma$.  Suppose that the limit of $\|\BSigma\|\|\BSigma^{-1}\|$ 
exists and is a constant $c\geq 1$. Then, assuming $n=2$ for ease, 
under the assumptions of Theorem~\ref{theorem:main} our required 
inequality for the condition number is
\[
\lim_{P,N\to\infty}\frac{ \|\BSigma\|\|\BSigma^{-1}\|\|\Har-\Id\|} 
{\|\Ave- \Id\|}= c\bigg(\frac{\sqrt{1-\frac{\gamma}{2}}} 
{1+\frac{1}{2}\sqrt{\frac{\gamma}{2}}}\bigg)<1,
\]
which is clearly non-vacuous, for instance when $\gamma = \frac{1}{2}$ 
the inequality requires
\[
c < \frac{5}{4}\sqrt{\frac{4}{3}}\approx 1.44337567\ldots.
\]
\end{remark}

In Section~\ref{sec:generalcovar} we provide the following fixed point
equation for the limiting Stieltjes transform of
$\BSigma^{\frac{1}{2}}\Har\BSigma^{\frac{1}{2}}$ assuming that
$\BSigma$ and $\Har$ as non-commutative random variables converge to a
pair of freely independent random variables (see
Section~\ref{sec:free}, Definition~\ref{def:freepois} and
equation~\eqref{eqn:defFhar} for relevant definitions and
terminology).

\begin{theorem}
\label{theorem:fixed}
Suppose that $(\BH,\BSigma)$ as a pair of non-commutative random 
variables converge in the sense of distribution to a pair 
$(\fHar,\fCovar)$ of non-commutative freely independent random 
variables with the law of $\fHar$ being the spectral measure defined in 
Theorem~\ref{theorem:main} and the law of $\fCovar$ being the limiting 
spectral measure of $\BSigma$ whose cdf we denote as $F$. We assume $F$ 
is supported on the positive reals. Then we have the following limiting 
fixed point equation for the Stieltjes transform of 
$\BSigma^{\frac{1}{2}}\BH\BSigma^{\frac{1}{2}}$, which we denote 
$m_{\fCovar\fHar}(z)$
\[
m_{\fCovar\fHar}(z) = \int_{\BBR^+} \frac{\diff F(x)}{z - x 
\big(\frac{\gamma z m_{\fCovar\fHar}(z)}{n} + 1-\gamma\big)},
\]
and the limiting fixed point equation for the Stieltjes transform of 
$\BSigma^{\frac{1}{2}}\BH\BSigma^{\frac{1}{2}}-\BSigma$, which we 
denote as $m_{\fError}(z)$, is
\[
m_{\fError}(z) = \int_{\BBR^+} \frac{\diff F(x)}{z - 
\frac{x}{S_{\breve\fHar}(zm_{\fError}(z) - 1)}},
\]
where $S_{\breve\fHar}(z)$ is the $S$-transform of $\fHar - 1_\cF$ 
which satisfies the quadratic:
\[
\frac{\gamma z}{n} S_{\breve\fHar}(z)^2 + \gamma\bigg(\frac{1+z}{n} -
1 \bigg) S_{\breve\fHar}(z) - 1 = 0.
\]
\end{theorem}

By Corollary~\ref{cor:submult}, it stands to reason that the 
improvement of the harmonic mean over the arithmetic mean in operator 
norm should be true for a wide range of covariance $\BSigma$. By the 
above fixed point characterization, we expect this improvement should 
only depend on the limiting distribution $\diff F$ of $\BSigma$. In 
future investigations we hope to characterize the role of $\diff F$ in 
the phenomenon described in Theorem~\ref{theorem:main} and 
Remark~\ref{rem:cutoff}. 

\subsection*{Outline}
The paper is organized as follows: Section~\ref{sec:free} provides 
relevant background terminology and results from free probability 
theory needed to understand the proof of Theorem~\ref{theorem:main} and 
Theorem~\ref{theorem:fixed}. Section~\ref{sec:opconvergence} states and 
proves Lemma~\ref{lem:keylemma}, which guarantees the operator norm 
convergence in Theorem~\ref{theorem:main}. 
Section~\ref{sec:proofofmain} gives the proof of 
Theorem~\ref{theorem:main}, which is reduced to a calculation when 
Lemma~\ref{lem:keylemma} is taken as given. 
Section~\ref{sec:generalcovar} gives the proof of 
Theorem~\ref{theorem:fixed}.
\section{Free Probability Theory}
\label{sec:free}

In order to prove the main results of this paper, we require some
tools from the theory of free probability. Free probability is a
generalization of classical probability invented by Dan Voiculescu in
the 1980s for the purpose of investigating some properties of operator
algebras \cite{V85}. We require this theory because the sequence of
$\{\BW_i\}_{i=1}^n$ given in Definition~\ref{def:matmodel} behave as
the ``joint law'' of a collection of non-commutative random variables
(see Definition~\ref{def:freepois}).  In Section~\ref{sec:proofofmain}
we will use this fact to directly compute the limiting spectral
measure of the harmonic mean $\Har$. Our primary references for the
exposition in this section are \cite[Chapter 5]{AGZ10} and
\cite[Chapters 1--7]{NS06}.

Let $(\cA, \|\cdot\|_\cA, *)$ denote a unital $C^*$-algebra with
involution $*$. This means $\cA$ is a complex vector space equipped
with a complete norm $\|\cdot \|_\cA$ (i.e., $\cA$ is a Banach space),
a bilinear product
\[	
(x,y)\in \cA\times \cA \mapsto xy \in \cA.
\]
and a unit element
\[
1_\cA \in \cA \hbox{ such that }1_\cA x = x, \hbox{ for all } x \in 
\cA.
\]
$\cA$ is a unital Banach algebra if in addition the norm satisfies
\begin{align*}
\|1_\cA\|_\cA &= 1,\\
\|ab\|_\cA &\leq \|a\|_\cA\|b\|_\cA.
\end{align*}
When $\cA$ has an involution operation
\[
a\in \cA \mapsto a^* \in \cA  \quad\hbox{and}\quad (a^*)^* = a
\]
which satisfies for all $a$, $b \in \cA$ and $\lambda \in \BBC$
\begin{align*}
(a + b)^* &= a^* + b^*, \\ (ab)^* &= b^*a^*,\\ (\lambda a)^* &= 
\bar{\lambda} a^*,\\ \|a^*a\|_\cA &= \|a\|^2_\cA,
\end{align*}
then we say that $\cA$ is a unital $C^*$-algebra. An element $a \in 
\cA$ of a $C^*$-algebra is invertible if there exists a $b$ such that 
$a b = b a = 1_\cA$. Notice that the algebraic structure of $\cA$ 
allows us to consider non-commutative polynomials over elements in 
$\cA$. The subalgebra of non-commutative polynomials in formal 
variables $x_1$, $\ldots$, $x_n$ will be denoted $\BBC\langle 
x_1,\ldots, x_n\rangle$.

If $\cA$ is a $C^*$-algebra, then for each $a \in \cA$ the spectrum of 
$a$ can be defined by
\[
\mathrm{spec}(a) = \{\lambda \in \BBC : (\lambda 1_\cA - a) \hbox{ is
not invertible}\},
\]
we can say an element in $\cA$ is non-negative, written $a \succeq_\cA 
0$, if $a^* = a$ and its spectrum is non-negative. Note that for the 
$C^*$-algebra $\mathrm{Mat}_P(\BBC)$ of $P$-by-$P$ matrices, this is 
identical to the definition of a positive-semidefinite matrix.

To apply free probability to our problem of interest, we need the 
notion of a $C^*$-probability space. A non-commutative 
$C^*$-probability space $(\cA, \|\cdot\|_\cA,*, \phi)$ is the unital 
$C^*$-algebra $(\cA,\|\cdot\|_\cA, *)$ equipped with a linear map
\[
\phi : \cA \mapsto \BBC
\]
satisfying $\phi(1_\cA) = 1$ and $\phi(a) \geq 0$ whenever $a 
\succeq_\cA 0$. Such a $\phi$ is called a {\em state}. If $\phi(ab) = 
\phi(ba)$ for every $a$, $b\in \cA$, then $\phi$ is called a {\em 
tracial} state. Finally, if for every $a \in \cA$
\[
\phi\big[(aa^*)\big]= 0 \implies a = 0,
\]
then $\phi$ is a {\em faithful} tracial state \footnote{For a faithful 
tracial state, the operator norm for any $a \in \cA$ can be recovered 
by taking a limit: \[ \lim_{k\to\infty} 
\phi\big((aa^*)^k\big)^{\frac{1}{2k}}=\|a\|_\cA, \] see 
\cite[Proposition 3.17]{NS06} for a proof.}.

Elements of $\cA$ are called non-commutative random variables, and for
any collection $a_1,\ldots, a_m \in \cA$, their joint law is the
map
\[
\mu_{a_1,\ldots,a_m}\big[Q(x_1,\ldots,x_m)\big] :=
\phi\big[Q(a_1,\ldots,a_m)\big],
\]
where $Q \in \BBC\langle x_1,\ldots,x_m\rangle$.

The most important $C^*$-probability space will be 
$(\mathrm{Mat}_P(\BBC),\|\cdot\|,*,\varphi_P)$ where
\begin{equation}
\label{eqn:trace}
\varphi_P(a):= \frac{1}{P} \Tr a.
\end{equation}
When $a$ is a normal matrix, $\varphi_P(a)$ is the integral over the 
normalized spectral measure of $a$:
\[
\varphi_P(a) = \int_\BBC x\bigg\{ \frac{1}{P}\sum_{i=1}^P
\delta_{\lambda_i(a)}(\diff x)\bigg\},
\]
where $\lambda_j(a) \in \BBC$ are the eigenvalues of $a$.

For non-commutative random variables, there is a notion of convergence 
in distribution as well as an analogue of independence called 
\emph{free} independence. Let $(\cA_m,\|\cdot\|_{\cA_m}, *, \phi_m)$ 
for $m\geq 1$ and $(\cA,\|\cdot\|,*,\phi)$ be a collection of 
non-commutative $C^*$-probability spaces. Suppose that for each $m$, 
$a_{m,1}$,$\ldots$, $a_{m,k} \in \cA_m$ is a collection of 
non-commutative random variables and let $a_1$, $\ldots$, $a_k \in \cA$ 
be a fixed collection of non-commutative random variables. We say 
$a_{m,1}$, $\ldots$, $a_{m,k}$ converge in distribution to $a_1$, 
$\ldots$, $a_k$ if for every non-commutative polynomial $Q \in 
\BBC\langle x_1,\ldots,x_k\rangle$,
\[
\lim_{m\to\infty}\mu_{a_{m,1},\ldots,a_{m,k}}\big[Q(x_1,\ldots,x_k)\big] 
= \mu_{a_1,\ldots,a_k}\big[Q(x_1,\ldots,x_k)\big].
\]
A sequence of non-commutative random variables $a_1$, $\ldots$, $a_k$
are freely independent if for any polynomials $Q_1$, $\ldots$, $Q_k$,
we have
\[
\phi\bigg[\prod_{i=1}^k\Big\{Q_i(a_i)-\phi\big[Q_i(a_i)\big]\Big\}\bigg]
=0,
\]
we say a sequence of non-commutative random variables $a_{m,1}$,
$\ldots$, $a_{m,k} \in \cA_m$ are asymptotically freely independent if
they converge in distribution to freely independent non-commutative
random variables $a_1$, $\ldots$, $a_m \in \cA$.

The random matrices $\BW_i$, when viewed as a sequence of random 
variables taking values in the $C^*$-probability space 
$(\mathrm{Mat}_P(\BBC),\|\cdot\|,*,\varphi_P)$, converge almost surely 
in the sense of distribution to a collection of non-commutative random 
variables $\fPois_1,\ldots,\fPois_n$:
\begin{equation}
\label{eqn:limitsoftrace}
\mu_{\BW_1,\ldots,\BW_n}\big[Q(x_1,\ldots,x_n)\big] \to 
\mu_{\fPois_1,\ldots,\fPois_n}\big[Q(x_1,\ldots,x_n)\big]\quad 
\mathrm{a.s.}
\end{equation}
We define the $\fPois_j$ and the state $\nu$ below.
\begin{definition}
\label{def:freepois}
Let $(\cF,\|\cdot\|_{\cF},*,\nu)$ be a $C^*$-algebra with faithful 
tracial state $\nu$ and non-commutative random variables $\fPois_1$, 
$\ldots\,$, $\fPois_n \in \cF$ that are self-adjoint, non-negative, 
freely independent and satisfy
\[
\nu(\fPois_i^k) = \int x^k \rho_{\mathrm{MP},\gamma}(\diff x),
\]
where $\rho_{\mathrm{MP},\gamma}$ is the Mar\v{c}enko-Pastur Law with 
parameter $\gamma$ defined in \eqref{eqn:mpdef}. The $\fPois_j$ are 
called \emph{free Poisson} non-commutative random variables.
\end{definition}
The $C^*$-probability space defined above is guaranteed to exist due
to a functional analytic construction called the free product
\cite[Section 5.2--5.3]{AGZ10}. In fact, it is easier for us to assume
we have this construction in hand for what follows
below. Specifically, there exists a Hilbert space $\cH$ and a
subalgebra $\cF$ in the space of bounded linear operators on $\cH$,
denoted $B(\cH)$, such that the $C^*$-algebra in Definition
\ref{def:freepois} is $\cF$ equipped with the operator norm and the
involution is the mapping that takes an operator to its
adjoint. Furthermore there is a $\zeta \in \cH$ such that
\[
\nu(\mathfrak{a}) = \langle \zeta, \mathfrak{a} 
\zeta\rangle\quad\hbox{for all}\quad \mathfrak{a} \in \cF.
\]
In particular, the spectral measure of each $\mathfrak{\fPois_i}$ is
$\rho_{\mathrm{MP},\gamma}$, see \cite[Theorem 5.2.24]{AGZ10}.

In the next section, we will use a result from \cite{CDM07} in addition 
to concentration results in \cite{V18,RV08} to show that the spectral 
measure of the harmonic mean $\Har$ converges to the law of the 
non-commutative random variable
\begin{equation}
\label{eqn:defFhar}
\fHar := n \big(\fPois_1^{-1} + \cdots + \fPois_n^{-1}\big)^{-1}.
\end{equation}
In addition, we will be able to show $\|\Har - \Id\|$ converges almost 
surely to $\|\fHar - 1_\cF\|_\cF$. First, however, we must establish 
the existence of $\fHar$.
\begin{lemma}
\label{lem:harmofpoiss}
The non-commutative random variable $\fHar$ in \eqref{eqn:defFhar} is 
well-defined and can be approximated by a sequence of non-commutative 
polynomials in $\{\fPois_1,\dots,\fPois_n\}$. 
\end{lemma}
\begin{proof}
A simple proof of this property comes directly from the fact that each 
of our $\fPois_j$ are represented as bounded linear operators on a 
Hilbert space $\cH$.  Since their spectral measure is 
$\rho_{\mathrm{MP},\gamma}$ which is supported on the positive reals, 
they are all invertible so each $\fPois_j$ is invertible and so is the 
sum
\[
\fPois_1^{-1} + \cdots + \fPois_n^{-1}.
\]
We may approximate $\fHar$ with non-commutative polynomials in 
$\fPois_i$ by utilizing the Neumann series. Let 
\[
\Delta := \max\big\{\|\fPois_i\|_\cF, \|\fPois_i^{-1}\|_\cF, 
\|\fHar\|_\cF, \|\fHar^{-1}\|_\cF\big\},
\]
now consider the partial sum of the geometric series
\[
\frac{1}{2\Delta}\sum_{k=0}^m
\bigg(1_\cF - \frac{\fHar^{-1}}{2\Delta}\bigg)^k
\]
by definition of $\Delta$, along with usual bounds on geometric series
we have
\[
\bigg\|\fHar - \frac{1}{2\Delta}\sum_{k=0}^m \bigg(1_\cF - 
\frac{\fHar^{-1}}{2\Delta}\bigg)^k\bigg\|_\cF= \bigg\|\fHar\bigg(1_\cF 
- \frac{\fHar^{-1}}{2\Delta}\bigg)^{m+1}\bigg\|_\cF \leq 
\frac{\Delta}{2^{m+1}},
\]
which goes to $0$ as $m\to\infty$. Similarly, each $\fPois_i^{-1}$ can
be expanded as the infinite series
\[
\fPois_i^{-1} = \frac{1}{2\Delta}\sum_{k=0}^\infty 
\bigg(1_\cF-\frac{\fPois_i}{2\Delta}\bigg)^k,
\]
with similar error bounds as the expansion for $\fHar$. Since 
$\fHar^{-1}$ is the sum of $\fPois_i^{-1}$ we need only insert the 
truncated geometric series of $\fPois_i^{-1}$ into the truncated 
geometric series for $\fHar$ to get a non-commutative polynomial in 
$\fPois_i$ that approximates $\fHar$ in the norm $\|\cdot\|_\cF$. 
\end{proof}

The polynomial approximation in the proof above will be used again in 
the next section and is the main technical ingredient in addition to 
Theorem~\ref{thm:strongconv} below to establish the operator norm 
convergence of $\Har$.

\section{Strong Convergence of the Harmonic Mean} 
\label{sec:opconvergence}
The following Theorem from \cite{CDM07} will be our main tool for 
obtaining explicit formulas for the limiting operator norm of $\Har- 
\Id$.

\begin{theorem}
\label{thm:strongconv}
Let $\{\BW_i\}_{i=1}^n$ satisfy Definition~\ref{def:matmodel}. Then 
$\BW_i$ are asymptotically free and converge in the strong sense to 
freely independent Poisson random variables $\fPois_1,\ldots, 
\fPois_n$. This means, for any fixed polynomial $Q \in \BBC\langle 
x_1,\ldots, x_n\rangle$, in addition to the convergence
\[
\frac{1}{P} \Tr Q(\BW_1,\ldots, \BW_n) \to 
\nu\big[Q(\fPois_1,\ldots,\fPois_n)\big] \quad \mathrm{a.s.}
\]
we have the convergence
\[
\|Q(\BW_1,\ldots,\BW_n)\| \to \|Q(\fPois_1,\ldots,\fPois_n)\|_\cF \quad
\mathrm{a.s.}
\]
\end{theorem}

In order for this theorem to imply our desired results, we will use the 
fact that $\Har$ can be approximated by polynomials in the matrices 
$\BW_i$. A concentration bound on the largest eigenvalues of both 
$\BW_i$ and $\BW_i^{-1}$ is necessary before proceeding. We prove this 
Lemma for matrices satisfying Definition~\ref{def:matmodel} and 
Definition~\ref{def:altmatmodel}.
\begin{lemma}
\label{lem:polyapprox}
Let $\{ \BW_i \}$ satisfy Definition~\ref{def:matmodel} or 
Definition~\ref{def:altmatmodel}. Then there exists a deterministic 
constant $\kappa>0$ that depends only on $n$ and the subgaussian 
parameter of the entries of $\BX_i$ such that the event 
\[
B_P := \big\{\max \big(\|\BW_i\|, \|\BW_i^{-1}\|, \|\Har\|, 
\|\Har^{-1}\|\big) > \kappa \big\}
\]
satisfies
\[
\sum_{P=1}^\infty\Prob(B_P) < \infty.
\]
\end{lemma}
\begin{proof}
For any $t > 0$:
\begin{multline*}
\Prob\big(\big\{\max\big(\|\BW_i\|, \|\BW_i^{-1}\|, \|\Har\|, 
\|\Har^{-1}\|\big)> t \big\}\big) \\ \leq  n\Prob\big(\{\|\BW_1\|> 
t\}\big) + n\Prob\big(\{\|\BW_1^{-1}\|> t \}\big) 
+\Prob\big(\{\|\Har\|> t\}\big) + \Prob\big(\{\|\Har^{-1}\|> t\}\big),
\end{multline*}
By the AMHM inequality in \eqref{eqn:matamhmineq}, we have 
\[
\{\|\Har \| > t \}\subset \{\|\Ave\| > t\}
\]
so the triangle inequality and union bound applied to $\|\Ave\|$ gives
\[
\Prob\big(\{\|\Har\| > t \}\big)\leq n\Prob\big(\{\|\BW_1\| > t\}\big).
\]
The triangle inequality and a union bound also yield
\[
\Prob\big(\{\|\Har^{-1}\| > t \} \big) \leq n 
\Prob\big(\{\|\BW_1^{-1}\| > t \}\big),
\]
so we have
\begin{multline}
\label{eqn:tailbounds}
\Prob\big(\big\{\max\big(\|\BW_i\|, \|\BW_i^{-1}\|, \|\Har\|, 
\|\Har^{-1}\|\big)> t \big\}\big) \\ \leq 2n\Prob\big(\{\|\BW_1\|> 
t\}\big) + 2n\Prob\big(\{\|\BW_1^{-1}\|> t \}\big).
\end{multline}
There are several methods to bound the first probability
on the right-hand side of \eqref{eqn:tailbounds}.
One is an $\epsilon$-net argument that is 
described in \cite[Thereom 4.4.5]{V18}, it gives a bound of the form
\begin{equation}
\Prob\bigg(\bigg\{\|\BW_i\| > C_1c_1\bigg(\sqrt{\frac{P}{N}} + 1 + 
t\bigg)^2\bigg\}\bigg) \leq 2 \exp(-Nt^2) \label{eqn:tailbound1}
\end{equation}
here $C_1$, $c_1>0$ only depend on the subgaussian parameter of the 
entries of $\BX_i$. The above bound is clearly summable for any $t > 0$ 
fixed. Note that the $\epsilon$-net argument given in \cite{V18} is for 
the model in Definition~\ref{def:altmatmodel}, a similar argument can 
easily be made for the model in Definition~\ref{def:matmodel} with 
limited adjustments.

We take more care to bound the second probability on the right-hand 
side of \eqref{eqn:tailbounds}. It suffices to bound the smallest 
singular value of $\BX_i$ since this is equal to 
$N^{1/2}\|\BW_i^{-1}\|^{-1/2}$. We consider the complex 
Gaussian model of Definition~\ref{def:matmodel} separately from the 
real-entried model of Definition~\ref{def:altmatmodel}.

For the model in Definition~\ref{def:matmodel} we use the fact that the 
eigenvalues of $\BW_i$ have the same distribution as the 
eigenvalues of $N^{-1}\BY\BY^*$ where $\BY$ is the lower-triangular matrix
\[
\BY = \begin{bmatrix}
D_{2N}     &        &     &      \\
L_{2(P-1)} & D_{2(N-1)} &  &\\
     &  \ddots      & \ddots &  \\
	 &		  &L_2 & D_{2(N-P+1)} 
\end{bmatrix}
\]
where each $D_j$ and $L_j$ in the above matrix are independent 
$\chi^2$-distributed random variable with $j$ degrees of freedom (see 
\cite{DE02} for a derivation). With this representation, the same
Gre\^sgorin disk argument that yields \cite[Equation (2)]{S85} yields the 
lower bound
\begin{multline*}
\|\BW_i^{-1}\|^{-1} \geq \min\bigg[\frac{1}{N} (D_{2N}^2 
-D_{2N}L_{2(P-1)}), \frac{1}{N} (L_2^2 + D_{2(N-P+1)}^2- 
D_{2(N-P+2)}L_2), \\ \min_{j\leq P-2}\frac{1}{N}\Big(L_{2(P-j)}^2 + 
D^2_{2(N-j)} - (D_{2(N-j+1)}L_{2(P-j)} 
+D_{2(N-j)}L_{2(P-j-1)})\Big)\bigg],
\end{multline*}
the rest of the arguments in \cite{S85} that bound from below the right 
hand side of the above expression carry through identically and yield 
a constant $\epsilon>0$ such that the event  $\{\|\BW_i^{-1}\| > 
\epsilon\}$ is summable.

For the model in Definition~\ref{def:altmatmodel} we use 
\cite[Theorem 1.1]{RV08} which states for any $\epsilon>0$,
\[
\Prob\Big(\Big\{\sqrt{N}\|\BW_i^{-1}\|^{-\frac{1}{2}} \leq \epsilon 
\big(\sqrt{N} - \sqrt{P-1}\big)\Big\}\Big)\leq (C_2\epsilon)^{N-P+1} + 
\exp(-c_2N)
\]
where $C_2$ and $c_2> 0$ depend only on the subgaussian parameter of the 
entries of $\BX_i$.  Rearranging yields
\[
\Prob\Bigg(\Bigg\{\|\BW_i^{-1}\| \geq \frac{1}{\epsilon^2\Big(1 - 
\sqrt{\frac{P-1}{N}}\Big)^2}\Bigg\}\Bigg)\leq (C_2\epsilon)^{N-P+1} + 
\exp(-c_2N)
\]
letting $\epsilon = 1/2C_2$ ensures that $\big(C_2\epsilon\big)^{N-P}$ 
is summable in $P$, since $\frac{P}{N}\to\gamma\in(0,1)$.

Combining these bounds, we can select $\kappa > 0$ large enough so that 
both tail bounds in \eqref{eqn:tailbounds} are summable.
\end{proof}

We will now use Lemma~\ref{lem:polyapprox}, Lemma~\ref{lem:harmofpoiss} 
and Theorem~\ref{thm:strongconv} to prove the strong convergence of of 
$\Har$ to the non-commutative random variable $\fHar$.

\begin{lemma}
\label{lem:keylemma}
Assume $\{\BW_i\}$ satisfy Definition~\ref{def:matmodel}, then the 
sequence of random matrices $\Har$ converge in distribution and in the 
strong sense to the non-commutative random variable $\fHar$.
\end{lemma}
\begin{remark}
\label{rem:generalization}
Note that the proof of this Lemma is restricted to matrices satisfying 
Definition~\ref{def:matmodel} only due to the application of 
Theorem~\ref{thm:strongconv} since Lemma~\ref{lem:polyapprox} was 
proven for the models in Definition~\ref{def:matmodel} and 
Definition~\ref{def:altmatmodel}. If Theorem~\ref{thm:strongconv} is 
extended to the models in Definition~\ref{def:altmatmodel}, then this 
Lemma would automatically apply and Theorem~\ref{theorem:main} would 
also extend to the matrices in Definition~\ref{def:altmatmodel}.
\end{remark}
\begin{proof}
We first show for any monomial $S \in \BBC\langle x_1 \rangle$ 
\[
\|S(\Har)\| \to \|S(\fHar)\|_{\cF},\quad \mathrm{a.s.}
\]
It suffices to prove the above convergence for the monomial $S(x) = x$ 
since for matrices $\BM_1$ and $\BM_2$:
\[
\|\BM_1^k - \BM_2^k\| \leq \|\BM_1 - \BM_2\| \sum_{j=0}^{k-1} 
\|\BM_1\|^{j}\|\BM_2\|^{k-j}
\]
for any $k\geq 1$ so the approximation argument we use below will carry 
through for general $S$.
  
By Lemma~\ref{lem:polyapprox}, on the event $B_P^c$, $\Har$ can be 
expanded as a Neumann series in $\Har^{-1}$:
\[
\Har = \frac{1}{2\kappa} \sum_{l=0}^\infty \bigg(\Id -  
\frac{\Har^{-1}}{2\kappa}\bigg)^l,
\]
Note that on $B_P^c$ the convergence rate of the partial sum is 
explicit and deterministic:
\[
\bigg\|\Har - \frac{1}{2\kappa} \sum_{l=0}^m \bigg(\Id - 
\frac{\Har^{-1}}{2\kappa}\bigg)^l\bigg\| = \bigg\|\Har \bigg(\Id - 
\frac{\Har^{-1}}{2\kappa}\bigg)^{m+1}\bigg\| \leq 
\frac{\kappa}{2^{m+1}}.
\]
The inverse of $\Har$ can also be expanded into a series,
\[
\Har^{-1} = \frac{1}{n} \sum_{j=1}^n \BW_j^{-1} =
\frac{1}{n}\sum_{j=1}^n\frac{1}{2\kappa}\sum_{l=0}^\infty \bigg(\Id -
\frac{\BW_j}{2\kappa} \bigg)^l,
\] 
and similar explicit deterministic convergence rates can be derived.
It follows that there is a sequence of non-commutative polynomials $Q_d 
\in \BBC\langle x_1,\ldots,x_n\rangle$, whose coefficients
depend only on $n$ and $\kappa$, such that on the event $B_P^c$,
\[
\| Q_d(\BW_1,\ldots,\BW_n) - \Har\| \leq \frac{1}{d}.
\]
This implies that on $B_P^c$, as $d \to \infty$,
$Q_d(\BW_1,\ldots,\BW_n)$ converges in operator norm to $\Har$.
Furthermore, for each $d$, we have
\[
\lim_{N,P\to\infty}  \|Q_d(\BW_1,\ldots, \BW_n)\| \to 
\|Q_d(\fPois_1,\ldots,\fPois_n)\|,
\]
with probability 1 by Theorem~\ref{thm:strongconv}. Note that if we 
select $\kappa$ larger than the value $\Delta$ defined in the proof of 
Lemma~\ref{lem:harmofpoiss} we also have the bounds
\[
\|Q_d(\fPois_1,\ldots, \fPois_n) - \fHar\|_\cF \leq \frac{1}{d},
\]
since the construction of the polynomial $Q_d$ in Lemma 
\ref{lem:harmofpoiss} is identical to the one described above and 
satisfies the same bounds when $\|\cdot\|$ is replaced by $\|\cdot 
\|_\cF$.

Next, we work on the event $B^c=\liminf B_P^c$, noting that the 
summability of $\Prob(B_P)$ implies $\Prob(\limsup B_P) = 0$. As $d\to 
\infty$, $Q_d(\fPois_1,\ldots, \fPois_n)$ converges to $\fHar$ in the 
norm $\|\cdot\|_\cF$. Triangle inequality implies
\begin{equation*}
\begin{split}
\big|\|\Har\| - \|\fHar\|_\cF\big| \leq& \big|\|\Har\| - 
\|Q_d(\BW_1,\ldots, \BW_n)\|\big| \\&+\big|\|Q_d(\BW_1,\ldots,\BW_n)\| 
- \|Q_d(\fPois_1,\ldots,\fPois_n)\|_\cF \big|\\ &+ 
\big|\|Q_d(\fPois_1,\ldots,\fPois_n)\|_\cF - \|\fHar\|_\cF\big|,
\end{split}
\end{equation*}
Since $B^c$ occurs with probability $1$, the first term is bounded by 
$\frac{1}{d}$ as $P\to\infty$ by construction of $Q_d$. 
The second term vanishes as $P\to \infty$ by Theorem 
\ref{thm:strongconv}. For any $\epsilon > 0$, there is a deterministic 
$d$ large enough that makes the third term smaller than $\epsilon$ in 
the above inequality. Therefore for arbitrary $d$ and $\epsilon > 0$:
\[
\limsup_{P\to\infty}\big|\|\Har\| - 
\|\fHar\|_\cF\big|\leq\frac{1}{d}+\epsilon,
\]
the result then follows.

The convergence 
\[
\frac{1}{P}\Tr S(\Har)  \to \nu\big(S(\fHar)\big)\quad \mathrm{a.s.}
\]
follows from a similar argument. Again without loss of generality 
assume $S(\Har) = \Har$, and write
\[
\begin{split}
\bigg|\frac{1}{P} \Tr \Har	 - \nu(\fHar)\bigg| \leq & 
\frac{1}{P}\big|\Tr\Har - \Tr Q_d(\BW_1,\ldots,\BW_n)\big|  \\ 
&+\bigg|\frac{1}{P}\Tr Q_d(\BW_1,\ldots,\BW_n) - 
\nu\big(Q_d(\fPois_1,\ldots,\fPois_n)\big)\bigg|\\ 
&+\Big|\nu\big(Q_d(\fPois_1,\ldots,\fPois_n)\big) - \nu(\fHar)\Big|.
\end{split}
\]
The first term is bounded by $\|\Har - Q_d(\BW_1,\ldots,\BW_n)\|$ which 
on $B^c$ is bounded by $\frac{1}{d}$ as $P \to \infty$. The second term 
goes to 0 with probability 1 as $P\to\infty$ by Theorem 
\ref{thm:strongconv}. By Lemma~\ref{lem:harmofpoiss} for any $\epsilon 
> 0$ there is a deterministic $d$ such that for $d$ large enough the 
third term is bounded by $\epsilon$.
\end{proof}
\section{Harmonic Mean of Free Poisson Random Variables}
\label{sec:proofofmain}

In Sections~\ref{sec:free} and~\ref{sec:opconvergence} we proved that 
the limiting spectral measure of $\Har$ is the law of the 
non-commutative random variable $\fHar$. Additionally, we proved 
\[
\|\Har - \Id\| \to \|\fHar - 1_\cF\|_\cF \quad\mathrm{a.s.}
\]
We can conclude the proof of Theorem~\ref{theorem:main} by computing 
the distribution of $\fHar$ and the value of $\|\fHar -1_\cF\|_\cF$, 
which follows from a now standard type of calculation from free 
probability theory, which is called \emph{additive free convolution}
\cite{V86}. 

Let $\sigma$ be a compactly supported probability measure on
$\BBR$. The Cauchy-Stieltjes transform of $\sigma$ is denoted
\[
m_{\sigma}(z) := \int_\BBR \frac{\sigma(\diff x)}{z-x}.
\]
Let $K_{\sigma}(z)$ be the functional inverse of $m_{\sigma}(z)$. 
We define the $R$-transform of $\sigma$ as
\[
R_{\sigma}(z) := K_{\sigma}(z) - \frac{1}{z}.
\]
For two compactly supported probability measures $\sigma_1$ and 
$\sigma_2$, on $\BBR$, the additive free convolution of $\sigma_1$ and 
$\sigma_2$, denoted $\sigma_1 \boxplus \sigma_2$, is the unique 
probability measure obtained by the relation
\[
R_{\sigma_1\boxplus\sigma_2}(z)= R_{\sigma_1}(z) + R_{\sigma_2}(z).
\]
The additive free convolution is significant in free probability
because if $\mu_a$ and $\mu_b$ are the laws of two freely independent
non-commutative random variables $a$ and $b$, respectively, then the
law of the non-commutative random variable $a+b$ is the measure $\mu_a
\boxplus \mu_b$.  Note that for notational ease in what follows, we
use $R_{a}$ to denote $R$-transform of the (compactly supported)
measure that is the law of the non-commutative random variable $a$.
Here, we use the additive free convolution to compute the law of
$\fHar$ via the following steps:
\begin{enumerate}

\item We use the Cauchy-Stieltjes transform of each $\fPois_i$ to
  compute the Cauchy-Stieltjes transform of $\fPois_i^{-1}$. The fixed
  point equation for the Cauchy-Stieltjes transform of $\fPois_i$ is a
  quadratic equation. This results in a fixed point equation for
  $\fPois_i^{-1}$ which is also a quadratic equation.

\item Using the definition of the $R$-transform above, we obtain a
  quadratic fixed point equation for the $R$-transform of each
  $\fPois_i^{-1}$.

\item Because each $\fPois_i$ is freely independent of any other
  $\fPois_j$ for $i \neq j$, $\fPois_i^{-1}$ is freely independent of
  $\fPois_j^{-1}$.  We may now compute the $R$-transform
\[
R_{\fPois_1^{-1} \boxplus \cdots \boxplus \fPois_n^{-1}}(z) = n
R_{\fPois^{-1}}(z)
\]
where $\fPois$ has the same law as all of the $\fPois_i$.

\item With the $R$-transform of $\fPois_1^{-1} + \cdots +
  \fPois_n^{-1}$ in hand, we compute the Cauchy-Stieltjes transform of
  $\fHar$. As a consequence of steps 1--3, this function satisfies a
  quadratic fixed point equation which can be solved.

\item We invert the Cauchy-Stieltjes transform of $\fHar$ using the
  usual Plemelj inversion formula. This gives the law of $\fHar$,
  which upon shifting by $1$ and using faithfulness of the state $\nu$
  yields the operator norm of $\fHar - 1_\cF$.
\end{enumerate} 

Our approach in the calculations outlined above is from the paper
\cite{RE07}, which provides a general framework for computing various
transforms for non-commutative random variables whose Stieltjes
transforms satisfy polynomial equations.  In our case, each
$m_{\fPois}$ satisfies the following fixed point equation \cite{MP67}
\begin{equation}
\label{eqn:quadmp}
\gamma z m_{\fPois}(z)^2 + m_{\fPois}(z)(1 - z - \gamma) +1 = 0.
\end{equation}

\begin{proof}[Proof of Theorem~\ref{theorem:main}]
Denote the Cauchy-Stieltjes transform of the law of each 
$\fPois_i$ by
\[
m_{\fPois_i}(z) := \int_{\BBR} \frac{\rho_{\mathrm{MP},\gamma}(\diff 
x)}{z-x}.
\]
To obtain the law of $\fHar$, we first compute the law of
\[
n\fHar^{-1}=\sum_{i=1}^n \fPois_i^{-1},
\]
which is the additive free convolution of the $n$ freely independent
random variables $\{ \fPois_i^{-1} \}$. Since they all have the same
parameter $\gamma>0$, we need only compute, for a fixed $\fPois$ with
the same law as $\fPois_i$, the $R$-transform
\begin{equation}
\label{eqn:rofhar}
R_{n \fHar^{-1}}(z) = n R_{\fPois^{-1}}(z).
\end{equation}
With the law of $n \fHar^{-1}$ in hand, we simply invert and rescale
to obtain the law of $\fHar$, which allows us to compute the value of
$\|\fHar - 1_\cF\|_\cF$.

The law of $\fPois^{-1}$ is the push-forward measure of 
$\rho_{\mathrm{MP},\gamma}$ by the mapping $x \mapsto \frac{1}{x}$.
We denote this measure by $\mu_{\fPois^{-1}}$.
Using the push-forward, we have 
\begin{align*}
m_\fPois(z) &= \int_{\BBR}\frac{\mu_{\fPois^{-1}}(\diff x)}{z 
-\frac{1}{x}} = \int_{\BBR} \frac{x\mu_{\fPois^{-1}}(\diff x)}{xz-1}\\ 
&= \int_\BBR\bigg\{\frac{1}{z}-\frac{1}{z^2}\frac{1}{\frac{1}{z} - 
x}\bigg\}\mu_{\fPois^{-1}}(\diff x)\\ &= \frac{1}{z} - \frac{1}{z^2} 
m_{\fPois^{-1}}\bigg(\frac{1}{z}\bigg),
\end{align*}
rearranging and replacing $z$ with $\frac{1}{z}$ yields
\begin{equation}
\label{eqn:firststeqn}
m_{\fPois}\bigg(\frac{1}{z}\bigg) = z - z^2 m_{\fPois^{-1}}(z).
\end{equation}

The fixed point equation \eqref{eqn:quadmp} for $m_\fPois$ can be 
rewritten as
\begin{equation}
\label{eqn:modfix}
\frac{\gamma}{z} m_\fPois\bigg(\frac{1}{z}\bigg)^2 + \bigg( 1 - \gamma -
\frac{1}{z}\bigg)m_\fPois\bigg(\frac{1}{z}\bigg) + 1 = 0,
\end{equation}
inserting \eqref{eqn:firststeqn} into \eqref{eqn:modfix} yields
\[
\gamma z\big(1-zm_{\fPois^{-1}}(z)\big)^2+\big(z(1-\gamma)-1\big)\big(1-
zm_{\fPois^{-1}}(z)\big)+1 = 0,
\]
which when simplified yields
\begin{equation}
\label{eqn:quadforpushforward}
\gamma z^2m_{\fPois^{-1}}(z)^2-m_{\fPois^{-1}}(z)\big[ 
z(1+\gamma)-1\big]+1=0.
\end{equation}

Equation \eqref{eqn:quadforpushforward} yields the following equation 
when for $K_{\fPois^{-1}}(z)$ 
\[
\gamma z^2K_{\fPois^{-1}}(z)^2-(1 + \gamma)zK_{\fPois^{-1}}(z)+z+1=0,
\]
substituting the $R$-transform $R_{\fPois^{-1}}(z)$ gives the equation
\[
\gamma z^2 \bigg(R_{\fPois^{-1}}(z) +\frac{1}{z}\bigg)^2 - (1+\gamma) z
\bigg(R_{\fPois^{-1}}(z) + \frac{1}{z}\bigg) + z + 1 = 0,
\]
and simplifying further gives
\begin{equation}
\label{eqn:quadforR}
\gamma z R_{\fPois^{-1}}(z)^2 + (\gamma - 1) R_{\fPois^{-1}}(z) + 1 = 0.
\end{equation}

Using the additive convolution formula \eqref{eqn:rofhar} in 
\eqref{eqn:quadforR} gives
\begin{equation}
\label{eqn:quadforRofhar}
\gamma zR_{n\fHar^{-1}}(z)^2+n(\gamma-1)R_{n\fHar^{-1}}(z) + n^2 = 0.
\end{equation}
We will solve for $m_{n\fHar^{-1}}$, by reversing the procedure we 
performed above to obtain the $R$-transform given the Stieltjes 
transform. Inserting the definition of the $R$-transform into 
\eqref{eqn:quadforRofhar} gives
\[
\gamma z\bigg(K_{n\fHar^{-1}}(z)-\frac{1}{z}\bigg)^2 + 
n(\gamma-1)\bigg(K_{n\fHar^{-1}}(z)-\frac{1}{z}\bigg) + n^2 = 0,
\]
simplifying this gives
\[
\gamma z^2K_{n\fHar^{-1}}(z)^2 +
\big\{(n-2)\gamma-n\big\}zK_{n\fHar^{-1}}(z)+n^2z-(n-1)\gamma+n=0,
\]
so that
\[
\gamma z^2 m_{n\fHar^{-1}}(z)^2 +\big\{(n-2)\gamma - n\big\} z 
m_{n\fHar^{-1}}(z) + n^2 \tilde m(z) - (n-1)\gamma + n= 0,
\]
changing variables $z \mapsto \frac{1}{z}$ gives
\begin{equation}
\label{eqn:fixinvhar}
\frac{\gamma}{z^2} m_{n\fHar^{-1}}\bigg(\frac{1}{z}\bigg)^2
+ \bigg\{\frac{(n-2)\gamma - n}{z} + n^2\bigg\}
m_{n\fHar^{-1}}\bigg(\frac{1}{z}\bigg)
- (n-1)\gamma + n= 0.
\end{equation}
As in the pushforward calculation that gave \eqref{eqn:firststeqn} 
before, we have the relationship
\[
m_{n\fHar^{-1}}\bigg(\frac{1}{z}\bigg)  = z - z^2 m_{n^{-1}\fHar}(z),
\]
which when substituted into \eqref{eqn:fixinvhar} gives
\[
\gamma(1 - zm_{n^{-1}\fHar}(z))^2 +\big\{(n-2)\gamma - n + n^2z\big\}(1 
- z m_{n^{-1}\fHar}(z)) - (n-1)\gamma + n= 0
\]
which simplifies to the quadratic
\[
\gamma z m_{n^{-1}\fHar}(z)^2 +n\{1-\gamma - n z\} 
m_{n^{-1}\fHar}(z)+n^2=0,
\]
rescaling the law of $n^{-1} \fHar$ gives the final equation for 
$m_{\fHar}$:
\begin{equation}
\label{eqn:harmonicmeanstieltj}
\frac{\gamma z}{n} m_{\fHar}(z)^2+\{1-\gamma-z\} m_{\fHar}(z)+1=0.
\end{equation}

The solution to the quadratic equation \eqref{eqn:harmonicmeanstieltj}
is
\begin{equation}
m_{\fHar}(z) = \frac{n\Big\{z - 1 + \gamma - \sqrt{ z^2 -2z(1-\gamma + 
\frac{2 \gamma}{n}) + (1-\gamma)^2}\Big\}}{2\gamma z},
\end{equation}
where the branch cut of the square root has been taken to be the 
positive real line. We have chosen this particular root of the 
quadratic due to the decay condition $m_{\fHar}(z) \sim \frac{1}{z}$ as 
$z \to \infty$ and the requirement that $m_{\fHar}(z)$ must be complex 
analytic off the real line. See, for example, \cite[\S2.4.3]{T10} for a 
more detailed calculation (for Wigner matrices) that explains the 
selection of the branch cut when solving fixed point equations for the 
Stieltjes transform. See \cite[\S3.3]{BS10} for a derivation of the 
MP-law using these techniques. To recover the law $\mu_{\fHar}$ from 
the above Stieltjes transform, we follow the usual inversion formula,
which appears in \cite[Theorem 2.4.3]{AGZ10},
\[
\lim_{y\to 0}-\frac{1}{\pi}\int_a^b \Im m_{\fHar}(x + iy) \diff x = 
\int_a^b \mu_\fHar(\diff x),
\]
where $a < b$ are continuity points of the measure $\mu_\fHar$. By 
computing directly, we get that $\mu_\fHar$ is absolutely continuous 
with respect to Lebesgue measure with density
\[
\frac{\diff \mu_\fHar}{\diff x} = \frac{n}{2\pi \gamma x} 
\sqrt{(e_+-x)(x-e_-)} \mathbf{1}_{[e_-,e_+]}(x)
\]
where $e_\pm$ are defined in Theorem~\ref{theorem:main}. Using 
faithfulness of the state $\nu$, we may conclude that the operator norm 
of $\fHar -1_\cA$ is the largest element in absolute value of the 
support of the measure $\mu_\fHar$ after it has been shifted to the 
left by one:
\[
e_\pm - 1 = -\gamma + \frac{2\gamma}{n} \pm 2 
\sqrt{\frac{\gamma}{n}}\sqrt{1 - \gamma + \frac{\gamma}{n}},
\]
the choice of $-$ sign makes the absolute value largest:
\begin{equation}
\|\fHar - 1_\cF\|_\cF = 
\gamma - \frac{2\gamma}{n} + 2\sqrt{\frac{\gamma}{n}}\sqrt{1 - \gamma
+ \frac{\gamma}{n}},
\end{equation}
this concludes the proof of Theorem~\ref{theorem:main}.
\end{proof}
\section{General Covariance Matrix}
\label{sec:generalcovar}

From the last Section, we know that $\Har$ converges in the strong
sense to a non-commutative random variable, $\fHar$, whose law we
computed in the previous section. As mentioned in the Introduction, we
can study the harmonic mean of general population $\BSigma$ by
multiplication $\BSigma^\frac{1}{2}\Har \BSigma^\frac{1}{2}$. In this
section, we will obtain a fixed point equation for both the limiting
spectral measure of $\BSigma^{\frac{1}{2}}\Har\BSigma^{\frac{1}{2}}$
and its centered version
$\BSigma^{\frac{1}{2}}\Har\BSigma^{\frac{1}{2}} - \BSigma$ in terms of
the limiting cdf $F$ of $\BSigma$ assuming $(\fHar,\BSigma)$ converge
as a set of non-commutative freely independent random variables
$(\fCovar,\fHar)$, where $\fCovar$ has law given by the measure $\diff
F$.

We use another tool from free probability called the 
\emph{multiplicative free convolution} \cite{V87}. To define the 
$S$-transform, for a non-commutative random variable $a$ in some 
non-commutative $C^*$-probability space $(\cA, \|\cdot\|,*, \phi)$ 
define the function
\[
g_a(z) := \sum_{n=1}^\infty \phi(a^n) z^n,
\]
we will assume the law $\mu_a$ of $a$ is a compactly supported measure
supported on $\BBR$. We have the relationship
\begin{equation}
\label{eqn:srel1}
g_a(z) = \frac{1}{z}m_a\bigg(\frac{1}{z}\bigg) - 1.
\end{equation}
Assume $\phi(a) \neq 0$ so that $\ell_a(z)$ is guaranteed to
exist and is the functional inverse of $g_a(z)$:
\[
\ell_a\big(g_a(z)\big) = g_a\big(\ell_a(z)\big) = z.
\]
The $S$ transform of a non-commutative random variable $a$ is defined as
\begin{equation}
\label{def:stransform}
S_a(z) := \frac{1+z}{z} \ell_a(z).
\end{equation}
For freely independent non-commutative random variables $a$ and $b$ 
with $\phi(a) \neq 0$ and $\phi(b) \neq 0$, we have the 
rule
\begin{equation}
\label{eqn:stransrule}
S_{ab}(z) = S_a(z) S_b(z).
\end{equation}
Supposing the law of both $a$ and $b$ are known, $\mu_a$ and $\mu_b$
respectively. We will derive a fixed point equation for the Stieltjes
transform $m_a(z)$ in terms using the formula
\eqref{eqn:stransrule}. First, note that \eqref{eqn:stransrule} can
be written as
\begin{equation}
\label{eqn:srel2}
\ell_{ab}(z) = \ell_a(z) S_{b}(z)
\end{equation}
replacing $z$ with $g_{ab}\big(\frac{1}{z}\big)$ gives
\[
\frac{1}{z} = \ell_a\bigg(g_{ab}\bigg(\frac{1}{z}\bigg)\bigg) 
S_b\bigg(g_{ab}\bigg(\frac{1}{z}\bigg)\bigg),
\]
now applying \eqref{eqn:srel1} to this yields
\[
\frac{1}{z} = \ell_a\big(zm_{ab}(z) - 1\big) S_{b}\big(zm_{ab}(z) -
1\big),
\]
rearranging yields
\[
\ell_{a}\big(zm_{ab}(z) - 1\big)= \frac{1}{z S_{b}\big(zm_{ab}(z) -
  1\big)} ,
\]
applying $g_a$ on both sides yields
\[
zm_{ab}(z) - 1 = g_a\Bigg( \frac{1}{z S_{b}\big(zm_{ab}(z) -
  1\big)}\Bigg),
\]
using \eqref{eqn:srel1} once more gives
\[
m_{ab}(z) = S_b\big(z m_{ab}(z) - 1\big)
m_a\Big(z S_{b}\big(zm_{ab}(z) -1\big)\Big),
\]
which written in integral form is:
\begin{equation}
\label{eqn:fixpoint1}
m_{ab}(z) =
\int_{\BBR} \frac{\diff F(x)}{z - \frac{x}{S_{b}(z m_{ab}(z) - 1)}}.
\end{equation}
We will use \eqref{eqn:fixpoint1} to prove Theorem~\ref{theorem:fixed}.

\begin{proof}[Proof of Theorem~\ref{theorem:fixed}] 
By assumption of the Theorem, it will suffice to study the law of
\begin{align}
  \label{def:ncvar}
  \fCovar^{\frac{1}{2}}&\fHar \fCovar^{\frac{1}{2}},\\
  \label{def:centncvar}
\fError := 
\fCovar^{\frac{1}{2}} \{&\fHar - 1_\cF\}\fCovar^{\frac{1}{2}},
\end{align}
where $\fCovar^{\frac{1}{2}}$ is the square root of $\fCovar$ which
exists because $\fCovar$ can be realized as a positive bounded
self-adjoint linear operator on a Hilbert space $\cH$.  For notational
ease, define
\[
\breve\fHar:= \fHar - 1_{\cF}.
\]

It is clear that the state $\nu$ is tracial since it is the limit in
distribution of the tracial state $\varphi_P$. Therefore, deriving the
law of the variables in \eqref{def:ncvar} and \eqref{def:centncvar}
is the same as deriving the law of
\[
  \fCovar\fHar \quad\hbox{and}\quad \fCovar(\fHar - 1_\cF),
\]
respectively. Furthermore, it is clear that $\nu(\fCovar) > 0$ and
by direct computation we have
\[
\begin{split}
\nu(\fHar) &= \int_{e_-}^{e_+} \frac{n\sqrt{(e_+ - 
x)(x-e_-)}}{2\pi\gamma} \diff x = \frac{n(e_+-e_-)^2}{2\pi\gamma} 
\int_0^1\sqrt{y}\sqrt{1-y} \diff y\\ &= \frac{n(e_+-e_-)^2}{2\pi\gamma} 
\frac{\Gamma\big(\frac{3}{2}\big)^2}{\Gamma(3)} = 
\frac{n(e_+-e_-)^2}{16\gamma} \neq 1,
\end{split}
\]
so both $\nu(\breve\fHar) \neq 0$ and $\nu(\fHar) \neq 0$. Hence we 
have the equations
\begin{align*}
S_{\fCovar \fHar}(z) &= S_{\fCovar}(z) S_{\fHar}(z),\\
S_{\fError}(z) &= S_{\fCovar}(z) S_{\breve{\fHar}}(z).
\end{align*}

We derive the fixed point equation for $\fCovar\fHar$ first.  From the
previous section,
\[
\frac{\gamma z}{n} m_\fHar(z)^2 + (1 - \gamma -z)m_\fHar(z) + 1  = 0,
\]
replacing $z$ with $\frac{1}{z}$ and applying \eqref{eqn:srel1} gives
\[
\frac{\gamma z}{n} (g_{\fHar}(z) +1)^2 + \big[z(1-\gamma) -
  1\big]\big(g_\fHar(z) + 1\big) + 1 =0,
\]
replacing $z$ with $\ell_\fHar(z)$ gives
\[
\frac{\gamma \ell_\fHar(z)}{n} (z + 1)^2 + 
\big[\ell_{\fHar}(z)(1-\gamma) - 1\big] \big(z + 1 \big) + 1 = 0,
\]
and solving for $\ell_\fHar(z)$ gives
\[
\ell_{\fHar}(z) = \frac{z}{z+1} \frac{1}{\frac{\gamma z}{n}+ 1 - 
\gamma\big(1-\frac{1}{n}\big)},
\]
which yields the simple formula
\[
S_{\fHar}(z) = \frac{1}{\frac{\gamma z}{n} + 1 - \gamma\big(1 - 
\frac{1}{n}\big)},
\]
applying equation \eqref{eqn:fixpoint1} gives the required result.

For the second limit equation, since
\[
m_{\breve\fHar}(z) = m_\fHar(z+1),
\]
it follows
\[
\frac{\gamma (z+1)}{n}m_{\breve\fHar}(z)^2-(\gamma 
+z)m_{\breve\fHar}(z)+1=0,
\]
replacing $z$ with $\frac{1}{z}$ and applying \eqref{eqn:srel1} yields
\[
\frac{\gamma z(z+1)}{n} (g_{\breve\fHar}(z)+1)^2 - (\gamma z + 
1)(g_{\breve\fHar}(z) + 1) + 1 = 0,
\]
replacing $z$ with $\ell_{\breve \fHar}(z)$ gives the polynomial 
\[
\frac{\gamma \ell_{\breve\fHar}(z)(\ell_{\breve\fHar}(z) + 1)}{n}
\big(z + 1\big)^2 - \big(\gamma \ell_{\breve\fHar}(z) + 1\big)(z+1)
+ 1 = 0
\]
rearranging this yields
\[
\frac{\gamma}{n} \ell_{\breve\fHar}(z)^2(z+1)^2 + \frac{\gamma}{n}
\ell_{\breve\fHar}(z)(1+z)^2 -\gamma \ell_{\breve\fHar}(z)(z+1) -z = 0
\]
inserting the definition of the $S$-transform in \eqref{def:stransform}
yields
\[
\frac{\gamma}{n}z^2 S_{\breve\fHar}(z)^2 + \frac{\gamma}{n}(1+z)z 
S_{\breve\fHar}(z) -\gamma z S_{\breve \fHar}(z) - z = 0
\]
since $z$ is a non-zero complex number, we can divide through by $z$ to
get
\[
\frac{\gamma z}{n} S_{\breve\fHar}(z)^2+\gamma\Big(\frac{1+z}{n} - 
1\Big)S_{\breve\fHar}(z) -1 = 0,
\]
which concludes the proof by another application of equation
\eqref{eqn:fixpoint1}.
\end{proof}

%\bibliographystyle{plain}
%\bibliography{hide}

\end{document}